\documentclass{svproc}

\usepackage{url}

\usepackage{amssymb}
\usepackage{amsmath}
\usepackage{graphicx}
\usepackage{tikz}
\usetikzlibrary{calc,intersections,through,backgrounds,arrows}

\begin{document}
\mainmatter
\title{Asymptotics for connected graphs and irreducible tournaments}

\titlerunning{Asymptotics for connected graphs and irreducible tournaments}
\author{Thierry Monteil\inst{1} \and Khaydar Nurligareev\inst{2}}
\authorrunning{Thierry Monteil and Khaydar Nurligareev}
\tocauthor{Thierry Monteil and Khaydar Nurligareev}
\institute{LIPN, CNRS (UMR 7030), Universit\'e Paris 13, F-93430 Villetaneuse, France,\\
\email{thierry.monteil@lipn.univ-paris13.fr},\\
\texttt{https://monteil.perso.math.cnrs.fr/}
\and
LIPN, CNRS (UMR 7030), Universit\'e Paris 13, F-93430 Villetaneuse, France,\\
\email{khaydar.nurligareev@lipn.univ-paris13.fr},\\
\texttt{https://lipn.univ-paris13.fr/\~{}nurligareev/}
}
\maketitle

\begin{abstract}
We compute the whole asymptotic expansion of the probability that a
large uniform labeled graph is connected, and of the probability that a large
uniform labeled tournament is irreducible. In both cases, we provide a
combinatorial interpretation of the involved coefficients.
\end{abstract}

\section{Introduction}

Let us consider the Erd\"os-R\'enyi model of random graphs $G(n,1/2)$, where for
each integer $n\geqslant 0$, we endow the set of undirected simple graphs on the set
$\{1,\ldots,n\}$ with the uniform probability: each graph appears with
probability $1/2^{n\choose 2}$.
The probability $p_n$ that such a random graph of size $n$ is connected goes to
$1$ as $n$ goes to $\infty$.
In 1959, Gilbert~\cite{Gilbert1959} provided a more accurate estimation and
proved that 
$$ p_n = 1 - \dfrac{2n}{2^{n}} + O\left(\dfrac{n^2}{2^{3n/2}}\right).$$

In 1970, Wright~\cite{Wright1970apr} computed the first four terms of the
asymptotic expansion of this probability:
$$ p_n = 1 - {n\choose1}\dfrac{1}{2^{n-1}} - 2{n\choose3}\dfrac{1}{2^{3n-6}} -
24{n\choose4}\dfrac{1}{2^{4n-10}} + O\left(\dfrac{n^5}{2^{5n}}\right).$$

The method can be used to compute more terms, one after another. However, it
does not allow to provide the structure of the whole asymptotic expansion,
since no interpretation is given to the coefficients $1,2,24,\dots$.

The first goal of this paper is to provide such a structure: the $k$th term of
the asymptotic expansion of $p_n$ is of the form 
$$i_k\ 2^{k(k+1)/2}\ {n\choose k}\ \dfrac{1}{2^{kn}},$$
where $i_k$ counts the number of irreducible labeled tournaments of size $k$.
A tournament is said \emph{irreducible} if for every partition $A\sqcup{B}$ of
the set of vertices there exist an edge from $A$ to $B$ and an edge from $B$ to
$A$.  Equivalently, a tournament is irreducible if, and only if, it is strongly
connected~\cite{Rado1943}\cite{Roy1958}.


\begin{theorem}[Connected graphs]\label{th-graph}
For any positive integer $r$, the probability $p_n$ that a random graph of size
  $n$ is connected satisfies
\begin{equation*}\label{formula: probability asymptotics for graphs}
 p_n = 1 - \sum\limits_{k=1}^{r-1} i_k {n\choose k}
  \dfrac{2^{k(k+1)/2}}{2^{nk}} + O\left(\dfrac{n^r}{2^{nr}}\right),
\end{equation*}
where $i_k$ is the number of irreducible labeled tournaments of size $k$.
\end{theorem}

In particular, as there are no irreducible tournament of size $2$, this
explains why there is no term in ${n\choose 2}\dfrac{1}{2^{2n}}$ in Wright's
formula.
This result might look surprising as it relates asymptotics of undirected
objects with directed ones.

A similar development happened for irreducible tournaments.
For $n\geqslant 0$, we endow the set of tournaments on the set $\{1,\ldots,n\}$ with
the uniform probability: each tournament appears with probability
$1/2^{n\choose 2}$.
In 1962, Moon and Moser~\cite{MoonMoser1962} gave a first estimation of the
probability $q_n$ that a labeled tournament of size $n$ is irreducible, which
was improved in \cite{Moon1968} into
$$q_n = 1 - \dfrac{n}{2^{n-2}} + O\left(\dfrac{n^2}{2^{2n}}\right).$$
In 1970, Wright~\cite{Wright1970jul} computed the first four terms of the
asymptotic expansion of the probability that a labeled tournament is
irreducible:
$$ q_n = 1 - {n\choose 1}2^{2-n} + {n\choose 2}2^{4-2n} - {n\choose
3}2^{8-3n} - {n\choose 4}2^{15-4n} + O\left(n^52^{-5n}\right).$$

Here again, we provide the whole structure of the asymptotic expansion,
together with a combinatorial interpretation of the coefficients (they are not
all powers of two):

\begin{theorem}[Irreducible tournaments]\label{th-tournaments}
For any positive integer $r$, the probability $q_n$ that a random labeled
  tournament of size $n$ is irreducible satisfies
\begin{equation*}\label{formula: probability asymptotics for tournaments}
q_n = 1 - \sum\limits_{k=1}^{r-1} \big(2i_k-i_k^{(2)}\big) {n\choose k} \dfrac{2^{k(k+1)/2}}{2^{nk}} + O\left(\dfrac{n^r}{2^{nr}}\right),
\end{equation*}
where $i_k^{(2)}$ is the number of labeled tournaments of size $k$ with two irreducible components.
\end{theorem}

We can notice that the coefficients cannot be interpreted as counting a single
class of combinatorial objects, since the coefficient $2i_2-i_2^{(2)} = 0 - 2$
is negative.

\section{Notations, strategy and tools}

Let us denote, for every integer $n$, $g_n$ the number of labeled graphs of
size $n$, $c_n$ the number of connected labeled graphs of size $n$, $t_n$ the
number of labeled tournaments of size $n$, and $i_n$ the number of irreducible
labeled tournaments of size $n$.
We have $p_n = c_n/g_n$ and $q_n = i_n/t_n$.

Looking for a proof of Theorem~\ref{th-graph}, we see two issues: finding a
formal relation between connected graphs and irreducible tournaments, and proving
the convergence.
A tool to settle the first issue is the symbolic method: we associate to each
integer sequence its exponential generating function:
$$G(z)=\sum\limits_{n=0}^{\infty}g_n\dfrac{z^n}{n!},\ \
C(z)=\sum\limits_{n=0}^{\infty}c_n\dfrac{z^n}{n!},\ \
T(z)=\sum\limits_{n=0}^{\infty}t_n\dfrac{z^n}{n!},\ \
I(z)=\sum\limits_{n=0}^{\infty}i_n\dfrac{z^n}{n!}.$$

Since $g_n = t_n = 2^{n\choose 2}$, we have
\begin{equation}\label{formula: G(z)=T(z)}
  G(z) = T(z) = \sum\limits_{n=0}^{\infty}2^{n\choose2}\dfrac{z^n}{n!}.
\end{equation}

Note that, while the number of labeled tournaments of size $n$ is equal to the
number of labeled graphs of size $n$, their associated species are not
isomorphic: for $n=2$, the two labeled tournaments are isomorphic (by swapping
the vertices), while the two labeled graphs are not, so this equality is
somewhat artificial.

Since every labeled graph can be uniquely decomposed as a disjoint union of
connected labeled graphs, we have
\begin{equation}\label{formula: G(z)=exp(C(z))}
  G(z) = \exp(C(z)).
\end{equation}

It remains to find a relation between tournaments and irreducible tournaments.
%

\begin{lemma}\label{lemma: tournament decomposition}
Any tournament can be uniquely decomposed into a sequence of irreducible
  tournaments.
\end{lemma}

In terms of generating functions, Lemma~\ref{lemma: tournament decomposition}
translates to

\begin{equation}\label{formula: T(z)=1/(1-I(z))}
  T(z) = \dfrac{1}{1-I(z)}.
\end{equation}

Hence, part of the work will be to let those expressions interplay.

Regarding asymptotics, we will rely on Bender's Theorem~\cite{Bender1975}:

\begin{theorem}[Bender]\label{theorem: Bender's}
 Consider a formal power series
 $$
  A(z) = \sum\limits_{n=1}^{\infty}a_n z^n
 $$
 and a function $F(x,y)$ which is analytic in some neighborhood of $(0,0)$. Define
 $$
  B(z) = \sum\limits_{n=1}^{\infty}b_n z^n = F(z,A(z))
 \qquad
 \mbox{and}
 \qquad
  D(z) = \sum\limits_{n=1}^{\infty}d_n z^n = \dfrac{\partial F}{\partial y}(z,A(z)),
 $$
 Assume that $a_n\ne0$ for all $n\in\mathbb{N}$, and that for some integer $r\geqslant1$ we have
 $$
 \mbox{ (i) }\quad
 \dfrac{a_{n-1}}{a_n}\to0\,\mbox{ as }\, n\to\infty;
 \qquad\qquad
 \mbox{ (ii) }\quad
 \sum\limits_{k=r}^{n-r}|a_ka_{n-k}|=O(a_{n-r})\,\mbox{ as }\, n\to\infty.
 $$
 Then
 $$
  b_n = \sum\limits_{k=0}^{r-1}d_ka_{n-k} + O(a_{n-r}).
 $$
\end{theorem}

\section{Proofs}

\begin{proof}[Proof of Lemma~\ref{lemma: tournament decomposition}]

Let $T$ be a tournament. It is either irreducible, and all is done, or it
  consists of two nonempty parts $A$ and $B$ such that all edges between $A$
  and $B$ are directed from $A$ to $B$. Applying the same argumentation
  recursively to $A$ and $B$, we obtain a decomposition of $T$ into a sequence
  of subtournaments $T_1,\ldots,T_k$, such that each $T_i$ is irreducible and
  for every pair $i<j$, all edges go from $T_i$ to $T_j$ (see
  Figure~\ref{picture: tournament decomposition}). Since $T_i$ are also the
  strongly connected components of $T$, the decomposition is unique.

\begin{figure}
\begin{center}
\begin{tikzpicture}[>= latex, line width=.5pt, scale=0.8]
 \begin{scope}[scale=0.8]
  \def\quot{0.1}  
  \coordinate [label=0:$1$] (a1) at (40pt,0);
  \coordinate [label=45:$2$] (a2) at ([rotate = 60] a1);
  \coordinate [label=135:$3$] (a3) at ([rotate = 120] a1);
  \coordinate [label=180:$4$] (a4) at ([rotate = 180] a1);
  \coordinate [label=225:$5$] (a5) at ([rotate = 240] a1);
  \coordinate [label=-45:$6$] (a6) at ([rotate = -60] a1);
  \coordinate (a32) at ($(a3)!\quot!(a2)$);
  \coordinate (a31) at ($(a3)!\quot/1.73!(a1)$);
  \coordinate (a36) at ($(a3)!\quot/2!(a6)$);
  \coordinate (a35) at ($(a3)!\quot/1.73!(a5)$);
  \coordinate (a34) at ($(a3)!\quot!(a4)$);
  \coordinate (a16) at ($(a1)!\quot!(a6)$);
  \coordinate (a15) at ($(a1)!\quot/1.73!(a5)$);
  \coordinate (a14) at ($(a1)!\quot/2!(a4)$);
  \coordinate (a21) at ($(a2)!\quot!(a1)$);
  \coordinate (a25) at ($(a2)!\quot/2!(a5)$);
  \coordinate (a24) at ($(a2)!\quot/1.73!(a4)$);
  \coordinate (a65) at ($(a6)!\quot!(a5)$);
  \coordinate (a64) at ($(a6)!\quot/1.73!(a4)$);
  \coordinate (a62) at ($(a6)!\quot/1.73!(a2)$);
  \coordinate (a54) at ($(a5)!\quot!(a4)$);
  \draw[->] (a4) -- (a54);
  \draw[->] (a2) -- (a62);
  \draw[->] (a4) -- (a64);
  \draw[->] (a5) -- (a65);
  \draw[->] (a1) -- (a21);
  \draw[->] (a4) -- (a24);
  \draw[->] (a5) -- (a25);
  \draw[->] (a4) -- (a14);
  \draw[->] (a5) -- (a15);
  \draw[->] (a6) -- (a16);
  \draw[->] (a1) -- (a31);
  \draw[->] (a2) -- (a32);
  \draw[->] (a4) -- (a34);
  \draw[->] (a5) -- (a35);
  \draw[->] (a6) -- (a36);
  \foreach \p in {a1,a2,a3,a4,a5,a6} \filldraw [black] (\p) circle (1.5pt);
 \end{scope}
 \begin{scope}[xshift=2.5cm]
  \draw (0,0) node {$\leadsto$};
 \end{scope}
 \begin{scope}[xshift=5.0cm, scale=0.9]
  \def\quot{0.1}  
  \def\r{24pt}  
  \def\h{7pt}  
  \coordinate (c4) at (-20pt,0);
  \draw (c4) circle (\r);
  \draw (c4) circle (\r);
  \draw (c4)+(0,\h) node {$4$};
  \coordinate (c5) at (60pt,0);
  \draw (c5) circle (\r);
  \draw (c5) circle (\r);
  \draw (c5)+(0,\h) node {$5$};
  \draw (140pt,0) circle (\r);
  \draw (140pt,0) circle (\r);
  \coordinate (c3) at (220pt,0);
  \draw (c3) circle (\r);
  \draw (c3) circle (\r);
  \draw (c3)+(0,\h) node {$3$};
  \foreach \i in {0,1,2} \draw[->] (\r-20pt+\i*80pt,0) -- ++(80pt-2*\r,0);
  \coordinate [label=0:$1$] (c1) at (153pt,0);
  \coordinate [label=180:$2$] (c2) at (135pt,10pt);
  \coordinate [label=180:$6$] (c6) at (135pt,-10pt);
  \coordinate (c16) at ($(c1)!\quot!(c6)$);
  \coordinate (c21) at ($(c2)!\quot!(c1)$);
  \coordinate (c62) at ($(c6)!\quot!(c2)$);
  \draw[->] (c6) -- (c16);
  \draw[->] (c1) -- (c21);
  \draw[->] (c2) -- (c62);
  \foreach \p in {c1,c2,c3,c4,c5,c6} \filldraw [black] (\p) circle (1.5pt);
 \end{scope}
\end{tikzpicture}
\end{center}
\caption{Decomposition of a tournament as a sequence of irreducible components.}\label{picture: tournament decomposition}
\end{figure}
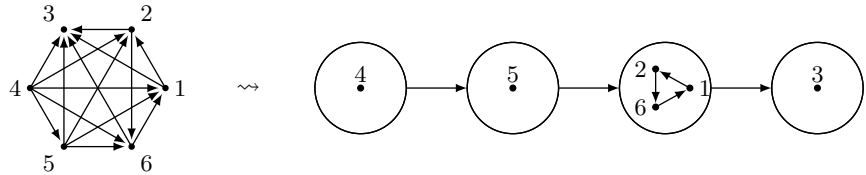

\end{proof}

\begin{proof}[Proof of Theorem \ref{th-graph}]

 Let us apply Bender's theorem (Theorem~\ref{theorem: Bender's}) the following way: take
$$
 A(z) = G(z) - 1
 \qquad\qquad
 \mbox{and}
 \qquad\qquad
 F(z,w) = \ln(1+w).
$$
Then, in accordance with formulas~\eqref{formula: G(z)=T(z)}, \eqref{formula: G(z)=exp(C(z))} and \eqref{formula: T(z)=1/(1-I(z))},
$$
 B(z) = \ln(G(z)) = C(z)
 \qquad
 \mbox{and}
 \qquad
 D(z) = \dfrac{1}{G(z)} = \dfrac{1}{T(z)} = 1 - I(z).
$$
Check the conditions of Theorem~\ref{theorem: Bender's}. In the case at hand, condition (i) has the form:
$$
 \dfrac{a_{n-1}}{a_n} = \dfrac{2^{n-1\choose2}}{(n-1)!}\dfrac{n!}{2^{n\choose2}} =  \dfrac{n}{2^{n-1}} \to 0,
 \qquad\mbox{as}\quad n\to\infty.
$$
To establish condition (ii), consider $x_k = n!a_ka_{n-k} = {n\choose k}2^{{k\choose2}+{n-k\choose2}}$, where $r\leqslant k\leqslant n-r$. Then $(x_k)$ decreases for $r\leqslant k\leqslant n/2$ and increases symmetrically for $n/2\leqslant k\leqslant n-r$.
  Bounding each summand by the first term is not enough, but bounding each summand (except for the first and last) by the second term gives the following:
 \begin{equation*}
 \begin{split} 
   \sum\limits_{k=r}^{n-r}a_ka_{n-k} & \leqslant \dfrac{1}{n!}{n\choose r}2^{{r\choose2}+{n-r\choose2}+1} + \dfrac{n-2r-1}{n!}{n\choose r+1}2^{{r+1\choose2}+{n-r-1\choose2}}\\
   & = O\left(\dfrac{2^{\big(n^2-(2r+1)n\big)/2}}{(n-r)!}\right) + O\left(\dfrac{2^{\big(n^2-(2r+3)n\big)/2}}{(n-r-2)!}\right) = O(a_{n-r}).
 \end{split}
 \end{equation*}
 Hence, Bender's theorem implies
 $$
  b_n = \dfrac{c_n}{n!} = \dfrac{2^{n\choose2}}{n!} - \sum\limits_{k=1}^{r-1}\dfrac{i_k}{k!}\dfrac{2^{n-k\choose2}}{(n-k)!}+O\left(\dfrac{2^{n-r\choose2}}{(n-r)!}\right).
 $$ 
 Dividing by $g_n/n!=2^{n\choose2}/n!$, we get
 $$
 p_n = \dfrac{c_n}{g_n} = 1 - \sum\limits_{k=1}^{r-1} i_k {n\choose k} \dfrac{2^{n-k\choose 2}}{2^{n\choose2}} + O\left(\dfrac{n^r}{2^{nr}}\right).
 $$

\end{proof}

\begin{proof}[Proof of Theorem \ref{th-tournaments}]

Let us apply Bender's theorem (Theorem~\ref{theorem: Bender's}) for
$$
 A(z) = T(z) - 1
 \qquad\qquad
 \mbox{and}
 \qquad\qquad
 F(z,w) = -\dfrac{1}{1+w}.
$$
Then, in accordance with formula~\eqref{formula: T(z)=1/(1-I(z))},
$$
 B(z) = -\dfrac{1}{T(z)} = - 1 + I(z)
 \qquad\qquad
 \mbox{and}
 \qquad\qquad
 D(z) = \dfrac{1}{\big(T(z)\big)^2} = \big(1 - I(z)\big)^2.
$$
Since $\big(I(z)\big)^2$ is the generating function for the class of labeled tournaments which can be decomposed into a sequence of two irreducible tournaments, we can rewrite the latter identity in the form
$$
 D(z) = 1 - \sum\limits_{n=1}^{\infty}\big(2i_k-i_k^{(2)}\big)\dfrac{z^n}{n!}.
$$
In the case at hand, the conditions that are needed to apply Theorem~\ref{theorem: Bender's} are the same as in the proof of Theorem~\ref{th-graph}, since the sequence $(a_n)$ is the same. Hence,
 $$
  b_n = \dfrac{i_n}{n!} = \dfrac{2^{n\choose2}}{n!} - \sum\limits_{k=1}^{r-1}\dfrac{2i_k-i_k^{(2)}}{k!}\dfrac{2^{n-k\choose2}}{(n-k)!}+O\left(\dfrac{2^{n-r\choose2}}{(n-r)!}\right).
 $$ 
 Dividing by $t_n/n!=2^{n\choose2}/n!$, we get
 $$
 q_n = \dfrac{i_n}{t_n} = 1 - \sum\limits_{k=1}^{r-1} \big(2i_k-i_k^{(2)}\big) {n\choose k} \dfrac{2^{n-k\choose 2}}{2^{n\choose2}} + O\left(\dfrac{n^r}{2^{nr}}\right).
 $$

\end{proof}

\section{Further results}

With a bit more work, we can compute the probability that a random graph of
size $n$ has exactly $m$ connected components,
and the probability that a random tournament of size $n$ has exactly $m$
irreducible components as $n$ goes to $\infty$. \\

In another direction, we can also generalize Theorem~\ref{th-graph} to the
Erd\"os-R\'enyi model $G(n,p)$, where the constant $2$ in the formulas is
replaced by $\rho=1/(1-p)$ and the sequence $(i_k)$ is replaced by a sequence
of polynomials $(P_k(\rho)) = 1, \rho-2, \rho^3-6\rho+6, \rho^6 - 8\rho^3
-6\rho^2 + 36\rho - 24, \dots$ with an explicit combinatorial interpretation.
\\

The methods presented here can also be extended to some geometrical context
where connectedness questions appear.
In particular, we will provide asymptotics for combinatorial maps, square tiled
surfaces, constellations, random tensor model
\cite{MonteilNurligareevPreparation}.
In some of the models, the coefficients in the asymptotic expansions show
connections with indecomposable tuples of permutations and perfect matchings.

\end{document}